\documentclass[a4paper,11pt]{amsart}

\usepackage{amssymb}
\usepackage{fullpage}
\usepackage{enumerate}
\usepackage[OT4]{fontenc}
\usepackage{tikz} 
\usepackage{float}
\usetikzlibrary{arrows}

\theoremstyle{definition}
\newtheorem{definition}{Definition}
\theoremstyle{plain}
\newtheorem{theorem}[definition]{Theorem}
\newtheorem{proposition}[definition]{Proposition}
\newtheorem{lemma}[definition]{Lemma}
\newtheorem{conjecture}[definition]{Conjecture}
\newtheorem{corollary}[definition]{Corollary}

\newtheorem{remark}[definition]{Remark}

\newtheorem{property}[definition]{Property}
\newtheorem{question}[definition]{Question}

\def\be{\begin{equation}}
\def\ee{\end{equation}}

\def\bp{\begin{property}}
\def\ep{\end{property}}

\def\bbP{{\mathbb P}}

\def\eee12{\frac{1}{\sqrt{r+\frac{1}{12}}}}

\def\ep{\varepsilon({\cal O}_{\bbP^2}(1), P_1,...P_r)}

\begin{document}

\title{\sc Prime numbers with a certain extremal type property}

\author{Edward Tutaj}



\begin{abstract}
The convex hull of the subgraph of the prime counting function
$x\rightarrow \pi(x)$ is a convex set, bounded from above by a
graph of some piecewise affine function $x\rightarrow
\epsilon(x)$. The vertices of this function form an infinite
sequence of points $(e_k,\pi(e_k))_1^{\infty}$. In this paper we
present some trivial observation about the sequence
$(e_k)_1^{\infty}$ and we formulate a number of questions
resulting from the numerical data. Besides we prove one less
trivial result: {\it if the Riemann hypothesis is true, then
$\lim\frac{e_{k+1}}{e_k}=1$.}

\end{abstract}

\maketitle

\section{Introduction}

Prime numbers are generators of the multiplicative semigroup
${\mathbb N}^{*}$ (where ${\mathbb
N}^{*}=\left\{1,2,3,...\right\}$). It is well known, that it is
impossible to distinguish two different prime numbers using only
the "language of multiplication". If one wants to distinguish some
particular prime number from the others, one must consider an
additional structure in ${\mathbb N}^{*}$, like for example  the
natural order in ${\mathbb N}$. The prime counting function is an
example of  such order properties. In this paper we define a
property of prime numbers with respect to their position on the
graph of the prime counting function $x\longrightarrow \pi(x)$.

Some properties related to the graph of the function $\pi$ were
studied se\-veral years ago in 1979 by Carl Pommerance \cite{Pomm}
and recently (2006) by H.L. Montgomery and S.Wagon \cite{MonWa} in
considerations concerning the Prime Number Theorem (PNT for
short).

Let $\mathbb P$ denote the sequence of prime numbers, i.e.
$\mathbb P = \left\{2,3,5,7,11,...\right\}$. Usually one defines
the function $\pi:[2,\infty)\longrightarrow [1,\infty)$ by the
formula

\be\label{wzor na pi} \pi(x)= \sum_{p\in \mathbb P, p\leq x}
1.
\ee

 For our purposes it will be a little more convenient to
consider a function $\pi^{*}:[2,\infty)\longrightarrow [1,\infty)$
defined as follows. First we define a continuous function $\eta:
[1,\infty)\longrightarrow [2,\infty)$ setting: $\eta(n)= p_n$,
where $p_n$ is the $n-th$ prime number, and $\eta$ is affine (and
continuous) in the intervals $[n,n+1]$ for each $n\in \mathbb N$.
Obviously $\eta$ is strictly increasing, continuous and
surjective. Thus $\eta$ is invertible and we define $\pi^{*}$ as
the inverse of $\eta$. Let $[x]$ denote the integral part of the
real number $x$. One can easily check, that $\pi$ and $\pi^{*}$
have the same values at prime numbers, and that

\be\label{wzor na pi*} \pi(x)= [(\pi^{*}(x))].
\ee
 \vspace{5mm}

\section{Part I}

\subsection{Definition of extremal primes.}

\vspace{3mm}

The function $\pi^{*}$ is increasing, continuous, but it is
"visibly" not concave. However there are many concave functions
$\varphi: [2,\infty)\longrightarrow [1,\infty)$, such that for
each $x\in [2,\infty)$ we have $\varphi(x)\geq \pi^{*}(x)$. This
follows for example from the Chebyshev  theorem, which gives the
inequality
\be\label{Czebyszew}
A\cdot\frac{x}{\ln(x)}<\pi(x)<B\cdot\frac{x}{\ln(x)}
\ee

for some $A<1$ and $B>1$, (obviously $\frac{x}{\ln x }$ is a
concave function).

\vspace{2mm}

Let us consider the set

\be
\Omega = \left\{f:[2,\infty)\longrightarrow [1,\infty): f\geq
\pi^{*}, f - {\rm  concave}\right\},
\ee and let us observe, although this will  play no role in our
consideration, that $\Omega$ is a subset of the vector cone of all
positive and concave real functions on $[2,\infty)$.

We put for $x\in [2,\infty)$

\be
\epsilon(x)= \inf\left\{f(x): f\in \Omega\right\},
\ee i.e. the function $\epsilon$ is the lower envelope of the
family $\Omega$. In other words the function $\epsilon$ is the
smallest concave function, which is greater than $\pi^{*}$
(equivalently than $\pi$). Since $\pi^{*}$ is piecewise affine,
then $\epsilon$ is also the lower envelope of those functions from
$\Omega$, which are piecewise affine. Then it is clear, that the
function $\epsilon$ is concave and it is also piecewise affine.
Thus the set

\be
\Gamma=\left\{(x,y)\in \mathbb R^{2}:x\in[2,\infty), 0\leq
y\leq \epsilon(x)\right\}
\ee is a convex set. Let us recall, that if $U$ is a convex set
and $b\in U$, then $b$ is said to be an {\it extremal point } of
$U$ iff $b$ is not an interior point of any non-trivial segment
lying in $U$.

Now we are ready to formulate the following:

\vspace{3mm}

\begin{definition}\label{definicja liczb ekstremalnych}
 {\it The prime number $p\in \mathbb P$ will be said to be extremal prime
number, when the point $(p,\pi(p))$ is an extremal point of the
convex set $\Gamma$}.
\end{definition}

\subsection{Properties of the set of extremal primes}

Let $\mathbb E$ denote the set of all extremal primes. Sometimes
we will think rather about the sequence of extremal primes
$\mathbb E = \left\{e_1,e_2,...,\right\},$ where $e_1<e_2<e_3...$,
i.e. the sequence $(e_k)_1^{\infty}$ is strictly increasing.

\vspace{3mm}

Now we will present some easy properties of the set $\mathbb E$.

\vspace{3mm}

\begin{proposition}
{\it The set $\mathbb E$ is not empty.}
\end{proposition}

\vspace{3mm}

Indeed, it is easy to check, that $2\in \mathbb E$. \vspace{2mm}

\begin{proposition}
 {\it The set ${\mathbb N}^{*}\setminus \mathbb E$ is not empty.}
\end{proposition}
\vspace{3mm}

One can check, that $3\in \mathbb E$, $7\in \mathbb E$, but
$5\notin \mathbb E$.

\vspace{3mm}

\begin{proposition}\label{rekurencja}
 {\it The set $\mathbb E$ is infinite}.
\end{proposition}

\vspace{3mm}

\begin{proof}
  Let $l_k$ denote the straight line (the affine function)
passing through the points $(e_{k-1},\pi(e_{k-1}))$ and
$(e_k,\pi(e_k))$. It follows from Definition 1 that the graph of
the function $\epsilon$ lies below the line $l_k$. This gives a
simple inductive method of finding the next extremal prime
$e_{k+1}$ providing, that we know $e_1, e_2,..., e_{k-1}, e_k$ (in
fact it is sufficient to know only $e_{k-1}$ and $e_k$). We can do
it as follows. We consider the difference quotients of the form

\be\label{iloraz1}
I_k(p)= \frac{\pi(p) - \pi(e_k)}{p-e_k},
\ee

for $p\in \mathbb P, p>e_k.$ It follows from the remark made
above, that for each $p>e_k$ we have: \be\label{iloraz2}
0<I_k(p)<\frac{\pi(e_k)-\pi(e_{k-1})}{e_k-e_{k-1}}=I_{k-1}(e_k).
\ee

Using the commonly known fact

\be\label{Legendre}
 \lim_{p\rightarrow\infty}\frac{\pi(p)}{p}=0
\ee we have
 $\lim_{p\rightarrow\infty} I_k(p)=0$.
Then there exists a finite set $\mathbb P_k\subset \mathbb P$ of
primes, such that $p_o\in \mathbb P_k \Longrightarrow p_o>e_k$ and
such that $I_k(p)\leq I_k(p_o)$ for $p>e_k$.  We set then
$e_{k+1}=\max {\mathbb {P}_k}$. This implies, that the set
$\mathbb E$ is infinite.

\end{proof}

\vspace{3mm}
\begin{proposition}\label{pochodna pi-e}

 {\it The derivative $x\longrightarrow {\epsilon}'(x)$ is
strictly decreasing and tends to 0 at infinity.}

\end{proposition}

\vspace{3mm}

\begin{proof}

Let

\be\label{definicja delta k}
\delta_k=\frac{\pi(e_{k+1})-\pi(e_k)}{e_{k+1}-e_k},
\ee i.e. $\delta_n$ is the slope of the n-th segment lying on the
graph of the function $\epsilon$. Since $\epsilon$ is increasing
and concave, then the sequence $(\delta_k)_{1}^{\infty}$ is
positive and  strictly decreasing. Let us observe, that the
sequence $(\delta_k)_{1}^{\infty}$ may be identified with the
derivative of the function $\epsilon$. Hence the limit $\delta =
\lim_{k\rightarrow \infty} \delta_k \geq 0$ exists and it must be
$\delta=0$, which follows once more from (\ref{Legendre}).

\end{proof}

The number $\alpha_k = {\delta_k}^{-1}$ is a measure of the
density of prime numbers in the interval $[e_k,e_{k+1})$ and may
be interpreted as an {\it average gap} between primes in
$[e_k,e_{k+1})$. By the remark made above, the sequence
$(\alpha_k)_1^{\infty}$ is strictly increasing.

\vspace{3mm}

It is natural to ask now about the cardinality of the set $\mathbb
N \setminus \mathbb E$. We have

\vspace{5mm}

\begin{proposition}\label{Zhang}
 {\it The set $\mathbb N\setminus \mathbb E$ is infinite.}
\end{proposition}

\begin{proof}

 This is true and is related to study of {\it small gaps between
 primes}. Let us observe only, that the finitness of $\mathbb N\setminus \mathbb E$
 is impossible if the {\it twin primes} conjecture is true. However, we
 know now from the recent result of Zhang, \cite{Zhang} that
  $\liminf(p_{n+1}-p_n)<7\cdot 10^7$. It follows from  Proposition \ref{pochodna pi-e}
  that this is  sufficient for
 the set $\mathbb N\setminus \mathbb E$ to be infinite.

\end{proof}

 It appears, that the set $\mathbb E$ is in some sense minimal with respect to
 Property \ref{pochodna pi-e}. Namely suppose, that $\mathbb G = (g_i)_1^{\infty}$ is a subsequence of
the sequence
  $\mathbb P$ of prime numbers
such that $g_1=2$. Let \be
\delta_k(\mathbb G)=\frac{\pi(g_{k+1})-\pi(g_k)}{g_{k+1}-g_k}.
\ee

We will say, that $\mathbb{G}$ is concave, when
$\delta_k(\mathbb{G})$ is strictly decreasing. For example the
sequence $\mathbb E$ is concave, while the sequence $\mathbb P$ is
not concave. A subsequence of a concave sequence is also concave.
The sequence $\mathbb E$ of extremal primes has the following
property: {\it if $\mathbb E$ is a subsequence  of a concave
sequence $\mathbb G$, then $\mathbb E = \mathbb G$.} More exactly:

\vspace{3mm}

\begin{proposition}\label{minimality}
{\it Let us suppose that  a sequence $(g_k)_{1}^{\infty}$ is
concave and the sequence $\mathbb E$ is a subsequence of $\mathbb
G$. Then $\mathbb E = \mathbb G$.}
\end{proposition}

\vspace{3mm}

\begin{proof}

 Clearly $e_1=g_1=2$. Since there are no primes between 2 and 3
and
 $e_2\in \mathbb G$ then also $e_2=g_2 =3$.
Suppose now that $e_i=g_i$ for $1\leq i \leq k$.
 We wish to prove, that $e_{k+1}=g_{k+1}$. Assume then, that $e_{k+1}\neq g_{k+1}$
and that $g_{k+m}=e_{k+1}$ i.e. that $$e_k=g_k <g_{k+1}< g_{k+2}<
...< g_{k+m} = e_{k+1}.$$

Now, using the notations from Proposition \ref{rekurencja} and the
definition of $e_{k+1}$ we have for $i<m$: \be
\delta_{k}(\mathbb G)=I_k(g_{k+1})<\delta_k(\mathbb E)
\ee

 Let us consider a function $H:[e_k,e_{k+1}]\longrightarrow
{\mathbb R}$
 such that $H(g_{k+i})= \pi(g_{k+i})$ and $H$ is affine and continuous in each
interval $[g_{k+i},g_{k+i+1}]$. We see, that the function $H$ is
continuous and differentiable except in the points $x=g_{k+i}$
 and its derivative in the
intervals $ (g_{k+i},g_{k+i+1})$ is constant and equal
$\delta_{k+i}(\mathbb G)$. It follows from our assumptions (since
$\mathbb G$ is concave), that

\be \sup\left\{H^{'}(x): x\in [e_k,e_{k+1}]\right\}=
\delta_{k}({\mathbb G}) < \delta_k(\mathbb E). \ee

 Let us observe, that since the function $H$
is continuous an differentiable except for a finite set of
arguments,  we can apply the mean value theorem. Hence we have:

\begin{eqnarray*}
\pi(e_{k+1}-\pi(e_k)= \pi(g_{k+p})-\pi(g_k)&\leq&
\sup\left\{(H^{'}(x): x\in [e_k,e_{k+1}]\right\}\cdot
(g_{k+p}-g_{k})\\&\leq &\delta_k(\mathbb G)\cdot(e_{k+1}-e_k)<
\delta_k(\mathbb E)\cdot(e_{k+1}-e_k) = \pi(e_{k+1})-\pi(e_k),
\end{eqnarray*}
 but this is impossible and this ends the
proof of  Proposition \ref{minimality}.
\end{proof}

\subsection{Some numerical data and the questions they evoke}

The observations about the extremal primes made above are rather
trivial. We will prove later some deeper, however conditional,
results. We have calculated the first 2200 extremal primes and
after studying these numerical data, we can formulate a number of
more or less interesting questions. It is impossible to give here
the complete list of the first 2200 extremal primes, but we will
present some selected data:

\vspace{3mm}


 The  first twenty eight  terms of the sequence $\mathbb E$ are:

   \begin{center}
   \renewcommand{\arraystretch}{1.2}
   \begin{tabular}{|c|c|c|c|c|c|c|c|c|c|c|c|c|c|c|}
   \hline
   $n$ & 1 & 2 & 3 & 4 & 5 & 6 & 7 & 8 & 9 & 10 & 11 & 12 & 13 & 14  \\
   \hline
   $e_n$ & 2 & 3 & 7 & 19 & 47 & 73 & 113 & 199 & 283 & 467 & 661 & 887 & 1129 & 1329  \\
   \hline
   \end{tabular}
   \end{center}

\begin{center}
   \renewcommand{\arraystretch}{1.2}
   \begin{tabular}{|c|c|c|c|c|c|c|c|c|c|c|c|c|c|c|}
   \hline
   $n$ & 15 & 16 & 17 & 18 & 19 & 20 & 21 & 22 & 23 & 24 & 25 & 26 & 27 & 28  \\
   \hline
   $e_n$ & 1627 & 2803 & 3947 & 4297 & 5881 & 6379 & 7043 & 9949 & 10343 & 13187 & 15823 & 18461 & 24137 & 33647  \\
   \hline
   \end{tabular}
   \end{center}

\vspace{3mm}


 The list of $e_k$ where $k\leq 2200$ and $k\equiv 0 (\mod 100)$:


   \begin{center}
   \renewcommand{\arraystretch}{1.2}
   \begin{tabular}{|c|r|}
   \hline
     $e_{100}$&  5253173  \\
   \hline
 $e_{200}$ & 67596937 \\
   \hline
 $e_{300}$ & 314451367\\

\hline
 $e_{400}$ & 883127303\\

\hline
 $e_{500}$ & 2122481761\\

\hline
 $e_{600}$ & 4205505103\\
\hline
 $e_{700}$ & 7274424463\\

\hline
 $e_{800}$ & 12251434927\\

\hline
 $e_{900}$ & 19505255383\\

\hline
 $e_{1000}$ & 28636137347\\

\hline
 $e_{1100}$ & 40001601779\\
\hline
 $e_{1200}$ & 55036621907\\

\hline
 $e_{1300}$ & 73753659461\\

\hline
 $e_{1400}$ & 97381385771\\
\hline
 $e_{1500}$ & 125232859691\\

\hline
 $e_{1600}$ & 157169830847\\
\hline
 $e_{1700}$ & 196062395777\\

\hline
 $e_{1800}$ & 241861008029\\
\hline
 $e_{1900}$ & 296478801431\\

\hline
 $e_{2000}$ & 365234091199\\
\hline
 $e_{2100}$ & 435006680401\\
\hline
 $e_{2200}$ & 524320812671\\
\hline

   \end{tabular}
   \end{center}

The examination of the sequence of the first 2200  extremal primes
allows us to formulate a number of questions. First of all it
seems to be interesting to say something about the "density" of
the sequence $\mathbb E$. Our "experimental" data support some
conjectures. Namely

\begin{conjecture}
{\it The series
$$\sum_{k=1}^{\infty}\frac{1}{e_k}$$ is convergent.}
\end{conjecture}

It follows from our data that
$$\sum_{k=1}^{2000}\frac{1}{e_k}\cong 1,090..$$

\begin{conjecture}
 {\it The series
$$\sum_{k=1}^{\infty}\frac{1}{\ln e_k}$$
is divergent.}
\end{conjecture}

Our data gives: $$\sum_{k=1}^{2000}\frac{1}{\ln e_k} > 100.$$

Since the set $\mathbb E$ of extremal  prime numbers is infinite
and, clearly,  the problem of finding any reasonable explicit
formula describing the correspondence $\mathbb N \ni
n\longrightarrow e_n$ is rather hopeless,  we may define and try
to study a function, which may be called  {\it extremal primes
counting function} $\pi_{\epsilon}$. The formula for
$\pi_{\epsilon}$ is analogous to the Formula  (\ref{wzor na pi}).
We set

\be\label{wzor na pi e}
\pi_{\epsilon}(x)= \sum_{p\in \mathbb E, p\leq x} 1.
\ee
 Unfortunately we know only $2200$ values of
$\pi_{\epsilon}(x)$ for $x\leq 5\cdot10^{11}$. However it seems to
be possible to formulate some conjectures about $\pi_\epsilon$.
Clearly $\pi_e(x)\leq \pi(x)$  and the growth of $\pi_\epsilon$ is
much  slower than the growth of $\pi$.
 For example $\pi_\epsilon(x_o)=1700$, when  $x_o = 196 062 395 777$ and for the same $x_o$ we have
 $\pi(x_o)= 7 855 721 212$. In particular we may try to find the
 best $\alpha < 1$ such that $\pi_{\epsilon}(x)=o(x^{\alpha})$
 observing the ratio $\frac{\ln n}{\ln e_n}$ when $n$ tends to
 infinity (in our case only to $n\leq 5\cdot 10^{11}$). May be
 only accidentally, but the best $\alpha$ obtained from our data
 is near to $\frac{\gamma}{2}$, where $\gamma$ is the Euler
 constant. Hence we formulate:

\begin{conjecture}
{\it There exists infimum

 $$\inf\left\{\alpha>0: \pi_{\epsilon}(x)=o(x^{\alpha})\right\}$$ and it is
 positive.}
\end{conjecture}

Our numerical data support strongly also the following interesting
conjecture:

\begin{conjecture}\label{conjectura 11}
 {\it In the notations as above, we have:
$$\lim_{k\rightarrow \infty}\frac{e_{k+1}}{e_k} = 1.$$}
\end{conjecture}

We will prove below, in Part II,  that the Riemann Hypothesis
implies the Conjecture \ref{conjectura 11}. This conjecture is
interesting itself, but also because of the following:

\begin{proposition}\label{PNT}
 {\it If $$\lim_{k\rightarrow \infty}\frac{e_{k+1}}{e_k} =
1$$then$$\lim_{n\rightarrow \infty}\frac{p_{n+1}}{p_n} = 1.$$}

\end{proposition}

\begin{proof}

 For each $n\in \mathbb N$ there exists $k(n)\in \mathbb N$ such
that
$$e_{k(n)}\leq p_n < p_{n+1}\leq e_{k(n)+1}.$$ Thus
$$\frac{p_{n+1}}{p_n}\leq \frac{e_{k(n)+1}}{e_{k(n)}}$$and the last sequence tends by our assumption to 1.
Let us recall here, that  $\lim_{n\rightarrow\infty}
\frac{p_{n+1}}{p_n} =1 $ implies PNT.

\end{proof}

It follows  directly from the definitions of the functions $\pi$
and $\pi_{\epsilon}$ that $\pi(e_{k+1})-\pi(e_k)\geq 1$ and the
equality may occur. Except for trivial $e_1=2$ and $e_2=3$ I have
found   two such "twin extremal primes" for $k=116$ and $k=976$.
Namely: $e_{116}=8 787 901$, $e_{117}= 8 787 917$ and
$\pi(e_{116})=589 274$, $e_{976}=26 554 262 369$ $e_{977}= 26 554
262 393$ and $\pi(e_{976})= 1 156 822 345$. We ask if:

\begin{question}\label{pytanie}
{\it Does there exists infinitely many $k\in \mathbb N$ such that
$\pi(e_{k+1})-\pi(e_k) = 1$.}
\end{question}

Some additional remarks about the "small" gaps between extremal
primes are in Part III.

Another exception is related to the inequality $I_k(p)\leq
I_k(p_o)$,
 which is described in  Proposition \ref{rekurencja}. One may ask if
the number of points $p>e_k$ such that $I_k(p)=I_k(p_o)$ is
greater than 1. In our numerical data we have only two such
examples, namely for $k=2$ we have $I_2(5)=I_2(7)$ and also
$I_4(23)=I_4(31)=I_4(43)=I_4(47) = \frac{1}{4}=\delta_4$ but in
fact our programme searching "next extremal primes" was not
written to "catch" such exceptions.

\section{Part II}

\subsection{Definition of lenses}

With the notation as in Part I, the intervals $[e_k,e_{k+1})$ (in
$\mathbb N$) will be called {\it lenses}. More exactly:

\begin{definition}
 Definition: {\it Given a positive integer $k\in \mathbb N$ the lens
 $S_k$ is a set $$S_k= \left\{n\in \mathbb N: e_k\leq n
 <e_{k+1}\right\}.$$ The difference $e_{k+1}- e_k$ will be called {\it the length} of the lens
 $S_k$ and will be denoted by $|S_k|$.}
 \end{definition}

 Sometimes we will use the name "lens" for a part of graph of
 $\pi^{*}$ for $x\in [e_k,e_{k+1})$. Our aim is to study the order
 of magnitude  of $|S_k|$ when $k\rightarrow \infty$. Since we
 will apply the language of differential calculus, it will be more
 comfortable to work with the function  $[2,\infty)\ni
 x\rightarrow S(x)\in [1,\infty)$ where
 $$x\in [e_k,e_{k+1})\Longrightarrow S(x)=|S_k|.$$
 The typical lenses and the graph of $\epsilon(x)$ for $x\leq 113$
 are illustrated on the pictures 1-3 at the end of this paper.

\subsection{The integral logarithm and error term}

We shall consider the following - well known -functions:
$L:[2,\infty)\longrightarrow [0,\infty)$ and
$\varepsilon:[2,\infty)\longrightarrow [0,\infty)$, defined by the
following formulas: \be\label{definicja Li}
L(x)=\int_{2}^{x}\frac{1}{\ln t}dt
\ee

 and
\be\label{definicja error}
\varepsilon(x)=\sqrt{x}\cdot\ln
x.
\ee

 The first is called {\it integral logarithm} (we will write
also $L(x)=Li(x)$), and the se\-cond is called {\it error term}.
Together with $L$ and $\varepsilon$ we will consider the functions

\be
\varphi(x)=L(x)-\varepsilon(x)
 \ee
 and for $x\in
(2,\infty)$ and $h\in \mathbb R$ \be l(x,h)= \varphi'(x)\cdot h +
\varphi(x)\ee 

Clearly all these functions are analytic at least in $(2,\infty)$.
We will use the derivatives of the considered functions  to the
order four and we shall write $y$ instead of $\ln x$  to present
some formulas in more compact form. Hence we have:

    \be\label{lp}
    L^{(1)}(x)= \frac{1}{\ln x}= \frac{1}{y}     
    \end{equation}

    \begin{equation}\label{ld}
    L^{(2)}(x)= \frac{-1}{x\cdot \ln x}= \frac{-1}{x\cdot
    y^2}                                        
    \end{equation}

    \begin{equation}\label{lt}
     L^{(3)}(x)= \frac{\ln x +2}{x^2\cdot {\ln^3 x}}
    = \frac{y+2}{x^2\cdot y^3},                
    \end{equation}

    \begin{equation}\label{lc}
    L^{(4)}(x)=\frac{-(2\cdot {\ln^2 x}+6 \ln x +6)}{x^3\cdot \ln^4 x}
    = \frac{-(2\cdot y^2 + 6y +6)}{x^3\cdot y^4}        
    \end{equation}


 The derivatives of error term function, written in an analogous manner, run as follows:

\be\label{ez}
    \varepsilon(x)= \sqrt{x}\cdot \ln x = \sqrt{x}\cdot y,
    \ee

\be\label{ep}
    \varepsilon^{(1)}(x)= \frac{\ln x
    +2}{2\sqrt{x}}=\frac{y+2}{2\sqrt{x}},
\ee

\be\label{ed}
    \varepsilon^{(2)}(x)= \frac{-\ln
    x}{4x\sqrt{x}}=\frac{-y}{4x\sqrt{x}}
\ee

\be
    \varepsilon^{(3)}=\frac{3\ln x - 2}{8x^2\sqrt{x}}=
    \frac{3y-2}{8x^2\sqrt{x}},
\ee

\be
    \varepsilon^{(4)}(x)=\frac{-15\ln x
    +16}{16x^3\sqrt{x}}=\frac{-15y+16}{16x^3\sqrt{x}}.
    \ee

Let us observe, that the second derivatives of the functions $L$
and $\varepsilon$ are negative, so both these functions are
concave.

The second derivative of the function $\varphi$ has the form
$$\varphi^{(2)}(x)=\frac{-4\sqrt{x}+\ln^3 x}{x\sqrt{x}\ln^2
x}=\frac{-4\sqrt{x}+y^3}{4x\sqrt{x}y^2}$$

then taking into account that
$$\lim_{x\rightarrow \infty}(-4\sqrt{x}+\ln^3 x)= -\infty$$
we can state :

\vspace{3mm}

\begin{proposition}\label{wypuklosc}
{\it There exists $x_o\in(2,\infty)$ such, that the function
$\varphi$ is concave in the interval  $[x_o,\infty)$.}
\end{proposition}

\subsection{A remark on Taylor polynomials of considered functions}

   Let us fix a point $x\in (2,\infty)$. Let $T^{(3)}_{x,L}$
 denote the Taylor polynomial of order three of the function
 $L$ with the center at $x$. Hence

\be
T^{(3)}_{x,L}(h)= L(x) + L^{(1)}(x)\cdot h+\frac{1}{2}\cdot
L^{(2)}(x)\cdot h^2 + \frac{1}{6}\cdot
L^{(3)}(x)\cdot h^3.
\ee

The remainder $R^{(3)}_x(h)=L(x+h)-T^{(3)}_{x,L}(h)$, written in
the Lagrange form,  is given by the formula:

\be R^{(3)}_x(h)= \frac{1}{24}L^{(4)}(\xi)\cdot h^4,\ee
 where
$\xi$ is a point from the  $(x,x+h)$. Since $L^{(4)}<0$ in all its
domain, we have the inequality:

\vspace{3mm}

\begin{proposition}
{\it  For each $x\in (2,\infty)$ and for each $h\in (2-x,\infty)$
the following inequality is true:
$$L(x+h)\leq T^{(3)}_{x,L}(h).$$}
\end{proposition}

Let $ T^{(3)}_{x,\varepsilon}$ denote the Taylor polynomial of
order three of the function  $\varepsilon$ with the center at $x$,
i.e. \be T^{(3)}_{x,\varphi}(h)= \varepsilon(x) +
\varepsilon^{(1)}(x)\cdot h+\frac{1}{2}\cdot \varphi^{(2)}(x)\cdot
h^2 + \frac{1}{6}\cdot L^{(3)}(x)\cdot h^3.\ee

Using an analogous argumentation as in the case of the function
$L$ we have:

\vspace{3mm}
\begin{proposition}\label{nierownosc z Taylorem}
 {\it  For each $x\in (2,\infty)$ and for each $h\in
(2-x,\infty)$ the following inequality is true:
$$\varepsilon(x+h)\leq T^{(3)}_{x,\varepsilon}(h),$$}
\end{proposition}

and in consequence we have the inequality (true for all $h\in
(2-x,\infty)$):

\be\label{rownanie glowne}
L(x+h)+\varepsilon(x+h)<T^{(3)}_{x,L}(h)+T^{(3)}_{x,\varepsilon}(h)
\ee

 \vspace{3mm}

\subsection{Definition of two functions}

 In this section we shall define two functions
$h_{+}:(x_o,\infty)\ni x \rightarrow h_{+}(x)\in \mathbb R$  and $
h_{-}:(x_o,\infty)\ni x\rightarrow h_{-}(x) \in \mathbb R$, where
$x_o$ is the point defined in Proposition \ref{wypuklosc}. First
we will describe in details the definition of the function
$h_{+}$. The definition of $h_{-}$ will be similar.

Let us fix a point  $x\in (x_o,\infty)$. Take into account the
tangent line $l(x,h)$ to the graph of the function  $\varphi$ at
the point $(x,\varphi(x))$. Its equation for $h\in \mathbb R$ is
given by:

\be\label{wzor 34}
 l(x,h)=\varphi'(x)\cdot h + \varphi(x)=
L'(x)h-\varepsilon'(x)h+L(x)-\varepsilon(x).
\ee

 The "tangent half-lines" obtained, when we restrict ourselves
in the Formula (\ref{wzor 34}) to $h\in [0,\infty)$ or $h\in
(-\infty,0]$ will be denoted by $l_{+}(x,h)$ or $l_{-}(x,h)$
respectively.

 For
$h=0$ we have the inequality:
$$l(x,0)=\varphi(x)=L(x)-\varepsilon(x)<L(x)+\varepsilon(x).$$

This means that the half-line $l_{+}$ "starts" from the interior
point $(x,\varepsilon(x))$ of the subgraph of the function
$L+\varphi$, which is a convex set. Since
$$\frac{d}{dh}L(x+h)=\frac{1}{\ln(x+h)}$$and
$$\frac{d}{dh}\varepsilon(x+h)=\frac{\ln(x+h)+2}{2\sqrt{x+h}}$$
then $$\lim_{h\rightarrow
\infty}\frac{d}{dh}(L(x+h)+\varepsilon(x+h)) = 0.$$

On the other hand  $$\frac{d}{dh}l(x+h) = \varphi'(x)
>0,$$hence the half-line $l_{+}(x,h)$ must intersect the graph of the strictly
concave function  $L(x+h)+\varepsilon(x+h)$ in exactly one point.
Hence we have proved the following:

\vspace{3mm}

\begin{proposition}
{ \it For each  $x\in (x_o,\infty)$ there exists exactly one
positive number $h_{+}(x)$ such that
$$L(x+h_{+}(x))+\varepsilon(x+h_{+}(x))=\varphi'(x)\cdot h_{+}(x) +
\varphi(x).$$}
\end{proposition}

In other words for each  $x\in (x_o,\infty)$ the equation (with
unknown $h$):

\be\label{rownanie33} L(x+h)+\varepsilon(x+h)=\varphi'(x)\cdot
h+\varphi(x)
\ee

 has exactly one positive solution, which we will denote by
$h_{+}(x)$.
\vspace{3mm}

If one replaces the half-line $l_{+}(x,h)$, by the half line
$l_{-}(x,h)$, then applying the same arguments as above, we
obtain:

\begin{proposition}
{ \it For each  $x\in (x_o,\infty)$ there exists exactly one
negative number $h_{-}(x)$ such that
$$L(x+h_{-}(x))+\varepsilon(x+h_{-}(x))=\varphi'(x)\cdot h_{-}(x) +
\varphi(x).$$}
\end{proposition}
In other words  equation (\ref{rownanie33}) has exactly one
negative solution, which we will denote by $h_{-}(x)$.

\subsection{An auxiliary equation}.

\vspace{3mm}

   In this paper we would like to establish the order of magnitude of the functions
 $x\rightarrow h_{+}(x)$ and $x\rightarrow h_{-}(x)$ (in fact of the difference
 $h_{+}(x)-h_{-}(x)$),
 when $x$ tends to $+\infty$. Since the equation (\ref{rownanie33}) is rather
 hard to solve,  we will consider
 an auxiliary equation:
\be\label{rownanie pomocnicze1}
T^{(3)}_{x,L}(h) +
T^{(3)}_{x,\varepsilon}(h)=\varphi'(x)\cdot h
+\varphi(x),
\ee
 which can be written in the form:
\be\label{rownanie pomocnicze2}
W_{x}(h):=\frac{1}{6}(L^{(3)}(x)+\varepsilon^{(3)}(x))\cdot h^3
 + \frac{1}{2}(L^{(2)}(x)+\varepsilon^{(2)}(x)) \cdot
h^2+2\varepsilon^{(1)}(x)\cdot h +2\varepsilon(x)=0.
\ee

 As we see,  equation (\ref{rownanie pomocnicze2}) is an algebraic equation of
degree three. It has at least one real root. We will see that it
can have (and has) more then one real root. We will be interested
not only on the existence of roots of  equation (\ref{rownanie
pomocnicze2}), but also on theirs signs.
 Let us observe, that since  $W_x(0)= 2\varepsilon(x)>0$ then the number  $h=0$
cannot be a root of considered equation. Let us also observe that,
in fact,  equation (\ref{rownanie pomocnicze2}) is not a single
algebraic equation, but it is a one parameter family of algebraic
equations, where the parameter is $x\in (x_o,\infty)$.

\vspace{3mm}

We will prove the following :

\vspace{3mm}

\begin{lemma}\label{Lemma1}

{\it i). There exists  $x_{+}\in (x_o,\infty)$, such that for each
$x>x_{+}$ the equation  $W_{x}(h)=0$ has a positive root.}

{\it ii). There exists  $x_{-}\in (x_o,\infty)$, such that for
each $x>x_{-}$ the equation  $W_{x}(h)=0$ has a negative root.}

\end{lemma}

The proof of the lemma is done together with the proof of
Proposition \ref{glownapropozycja}. Assume now, that Lemma
\ref{Lemma1} is true. This allows us to define two new functions
$h^{*}_{+}$ and $h^{*}_{-}$. We will describe in details the
definition of $h^{*}_{+}$. We set

\vspace{3mm}
\begin{definition}
{ \it Let $x\in (x_{+},\infty)$. Then the set of positive roots of
 equation  (\ref{eq35}) is not empty and we
set:$$h^{*}_{+}(x)= \min\left\{h>0: W_x(h)=0\right\}.$$}
\end{definition}
The relation between the functions $h_{+}$ and $h^{*}_{+}$ is the
following:

\vspace{3mm}
\begin{proposition}\label{propozycja22}
  {\it If Lemma \ref{Lemma1} is
true, then for $x\in (x_{+},\infty)$ we have the inequality:
$h_{+}(x)<h^{*}_{+}(x)$.}
\end{proposition}
\vspace{3mm}

\begin{proof}
 Let us fix   $x\in (x_{+},\infty)$. In the interval
$[x,x+h_{+}(x)]$, i.e. for $h\in [0,h_{+}(x)]$ the line  $l(x,h)$
lies below the graph of the function  $L+\varepsilon$. This
follows directly from the definition of the function  $h_{+}(x)$.
Hence in this interval the line $l(x,h)$ cannot intersect the
graph of the function  $T^{(3)}_{x,\varepsilon} + T^{(3)}_{x,L}$
because of  inequality (\ref{rownanie glowne}). Hence the equation
$W_x(h)=0$ has no roots in the interval $h\in [0,h_{+}(x)]$. But
this means that
 $h_{+}(x)<h^{*}_{+}(x)$, which ends the proof of  Proposition
 \ref{propozycja22}.
\end{proof}

\vspace{3mm}

Assume once more, that Lemma \ref{Lemma1} is true. We have

\vspace{3mm}

\begin{definition}
{ \it Let $x\in (x_{-},\infty)$. Then the set of negative roots of
 equation (\ref{rownanie pomocnicze2})  is not empty and we set:$$h^{*}_{-}(x)=
\max\left\{h<0: W_x(h)=0\right\}.$$}
\end{definition}

\vspace{3mm}

The relation between the functions $h_{-}$ and $h^{*}_{-}$ is as
follows:

\vspace{3mm}
\begin{proposition}\label{proposition 24}

 {\it If Lemma (\ref{Lemma1}) is
true, then for $x\in (x_{-},\infty)$ we have the inequality:
$h_{-}(x)>h^{*}_{-}(x)$.}

\end{proposition}

\vspace{3mm} The proof  of  Proposition \ref{proposition 24} is
similar to the proof of Proposition \ref{propozycja22}.

\vspace{3mm}

\subsection{The proof of the main lemma}

 Now we will prove Lemma (20).  Equation
 (\ref{rownanie pomocnicze2})
 we are interested in, can be written in the form:

\be\label{rownanie z A} A_3(x)\cdot h^3 + A_2(x)\cdot h^2 +
A_1(x)\cdot h +
A_o(x)=0
\ee
 where, using  formulas 21-28, we have:

\be A_3(x)=
\frac{1}{6}(L^{(3)}(x)+\varepsilon^{(3)}(x))=\frac{1}{48}\cdot\frac{8\sqrt{x}(y+2)+y^3(3y-2)}
{x^2\sqrt{x}y^3}, \ee \be
A_2(x)=\frac{1}{2}(L^{(2)}(x)+\varepsilon^{(2)}(x))=
\frac{-1}{8}\cdot\frac{4\sqrt{x}+y^3}{x\sqrt{x}y^2}.\, \ee \be
A_1(x)= \frac{y+2}{\sqrt{x}}, \ee

\be A_o(x)=2\sqrt{x} y.\ee

 Now, taking into account the fact, that for $x$ sufficiently
 large
$A_3(x)>0$, we divide  equation (\ref{rownanie z A})
 by  $A_3(x)$
in order to obtain the form:

\be\label{rownanie z B}
h^3+ B_2(x)\cdot h^2 +B_1(x)\cdot h + B_o(x)=0 
\ee

where

\be B_2(x)=\frac{A_2(x)}{A_3(x)}=-6x\frac{4\sqrt{x}y+y^4}
{8\sqrt{x}y+16\sqrt{x}+3y^4-2y^3},
 \ee

\be B_1(x)= \frac{A_1(x)}{A_3(x)}=
48x^2\frac{y^3}{8\sqrt{x}y+16\sqrt{x}+3y^4-2y^3}, \ee

\be B_o(x)=\frac{A_o(x)}{A_3(x)}=
96x^3\frac{y^4}{8\sqrt{x}y+16\sqrt{x}+3y^4-2y^3}.
 \ee

For further analysis of  equation \ref{rownanie z B} it will be
convenient to use some Landau symbols. Let us recall that for a
function  $g$ defined in the neighbourhood of  $+\infty$ one
writes  $g=o(1)$ if and only if  $\lim_{x\rightarrow
+\infty}g(x)=0$. Using this convention, we can write:

\be
B_2(x)=-6x\frac{\frac{1}{2}+o(1)}{1+o(1)},
\ee

\be B_1(x)=48x^2\frac{o(1)}{1+o(1)}, \ee

\be B_o(x)=96x^3\frac{o(1)}{1+o(1)}. \ee

This makes possible to write  equation \ref{rownanie z B} in the
form:

\be h^3 - 6x\frac{\frac{1}{2}+o(1)}{1+o(1)}h^2
+48x^2\frac{o(1)}{1+o(1)}h+96x^3\frac{o(1)}{1+o(1)}=0.\ee

Now we apply the substitution  $h=\theta x$, which leads to the
form: \be\label{eq50}
\theta^3x^3-6x\frac{\frac{1}{2}+o(1)}{1+o(1)}\theta^2x^2
+48x^2\frac{o(1)}{1+o(1)}\theta x + 96x^3\frac{o(1)}{1+o(1)}=0.\ee

Since we work only with $x>0$,  we can divide the last equation by
$x^3$, and we obtain the following equation (with unknown $
\theta$):

\be\label{eq50} \theta^3-6\frac{\frac{1}{2}+o(1)}{1+o(1)}
\theta^2+48\frac{o(1)}{1+o(1)}\theta+96\frac{o(1)}{1+o(1)}=0.
\ee

 Finally, taking into account the equality:
$$\frac{\frac{1}{2}+o(1)}{1+o(1)}=\frac{1}{2}+o(1)$$
we can write equation (\ref{eq50}) in the form: \be\label{eq51}
\theta^3 - 3\theta^2 +
v_{2}(x)\theta^2 +v_{1}(x)\theta +v_{o}(x)=0,
\ee
 where $v_1(x)$, $v_2(x)$, $v_o(x)$ are three positive functions
defined in a neighbourhood of $+\infty$ and tending to 0 when $x$
tends to $+\infty$. If for a fixed $x'$ we  find a number
$\theta'$ being a root of  equation (\ref{eq51}), then the number
$h'=\theta'\cdot x'$ is a root of  equation (\ref{rownanie z B}).
It is then enough to study  equation (\ref{eq51}). We shall prove
much more. Namely we have the following:

\begin{proposition}\label{glownapropozycja}
  For each   $\alpha>0$ there exists a point $x_{2}$ such
that for each $x>x_2$  equation (53) has in the interval
$[-\alpha,\alpha]$ exactly two roots $\theta_{-}$ and
$\theta_{+}$, and moreover $\theta_{-}<0<\theta_{+}$.
\end{proposition}

\vspace{3mm}

\begin{proof}
 Indeed, Proposition \ref{glownapropozycja} is  stronger than Lemma \ref{Lemma1}, where we need only the
existence of a negative root and of a positive root. In
Proposition \ref{glownapropozycja} we prove not only that the
roots exist, but also that we can find the solutions in an
arbitrary open interval containing the origin. Without loss of
generality, we may assume, that $\alpha\leq 1$. Let us fix then a
positive number $1\geq\alpha>0$ and choose $x_2$ so large, that
for $x>x_2$ we have:

\be\label{in52} v_2(x)\cdot \alpha^2+v_1(x)\cdot \alpha +
v_o(x)<2\alpha^2\ee             
and \be\label{in53}  v_2(x)\cdot \alpha^2-v_1(x)\cdot \alpha +
v_o(x)
<2\alpha^2,\ee

Such an $x_2$ exists since all three functions $v_2$, $v_1$, $v_o$
are $o(1)$ when $x$ tends to $+\infty$. Let us fix $x>x_2$. We
rewrite  equation (\ref{eq51}) in the form: $f(\theta)=g(\theta)$,
where

\be f(\theta)= \theta^3 + v_2(x)\cdot \theta^2+v_1(x)\cdot \theta
+ v_o(x),\ee
and

\be\label{eq55} g(\theta)=3\cdot \theta^2.\ee

 Let us set
$h(\theta)=f(\theta)-g(\theta)$ and let us consider the interval
$[0,\alpha]$. We have: $h(0)= f(0)-g(0)=v_o(x)
>0$ and , (since $\alpha<1$ and using the inequality (52))we obtain:
$$h(\alpha)=f(\alpha)-g(\alpha)= \alpha^3 + v_2(x)\cdot \alpha^2+v_1(x)\cdot
\alpha + v_o(x)<\alpha^2+2\alpha^2-3\alpha^2=0.$$ Thus  equation
(\ref{eq51}) has a root $\theta_{+}\in (0,\alpha)$.

Now we will consider the interval $[-\alpha,0]$. For $\theta=0$ we
have, as above $h(0)=v_o(x)>0$. For $\theta=-\alpha$ we have
(since $-\alpha^3<0$ and we have inequality (\ref{in53}):

\be
h(-\alpha)=f(-\alpha)-g(-\alpha)= -\alpha^3 + v_2(x)\cdot
\alpha^2-v_1(x)\cdot \alpha + v_o(x)- 3\alpha^2 <\ee

\be <v_2(x)\cdot \alpha^2-v_1(x)\cdot \alpha +
v_o(x)-3\alpha^2<2\alpha^2-3\alpha^2<0.\ee

Once more  the continuity argument implies the existence of the
root $\theta_{-}$ of the equation (\ref{eq51}) in the interval
$(-\alpha,0)$. Let us remark, that $\theta_{-}\cdot
x=h^{*}_{-}(x)$ and $\theta_{+}\cdot x=h^{*}_{+}(x)$. This ends
the proof of Proposition \ref{glownapropozycja}, hence moreover
Lemma \ref{Lemma1}.

\end{proof}

\vspace{3mm}

\subsection{The order of magnitude of lenses}

\vspace{3mm}

 By the results of the previous subsection, we can consider four functions:
$h_{-}$, $h_{+}$,$h^{*}_{-}$ and $h^{*}_{+}$, which are defined in
an interval $(M,\infty)$, and such that the following inequalities
holds (for each $x\in(M,\infty)$) :

\be h^{*}_{-}(x)<h_{-}(x)<0<h_{+}(x)<h^{*}_{+}(x).\ee
 Our aim is
to establish the order of magnitude at $+\infty$ of the difference
$H(x)=h_{+}(x)-h_{-}(x)$. We will prove the following:

\vspace{3mm}

\begin{proposition}\label{Prop o male}
 {\it The function $H$ satisfies the relation:
$$H(x) = o(x),$$when $x$ tends to $+\infty$}
\end{proposition}
\vspace{3mm}

\begin{proof}
 This follows directly from the property formulated in
Proposition \ref{glownapropozycja}. Indeed, it is sufficient to
show separately,
 that $h_{+}(x)=o(x)$ and $|h_{-}(x)|=o(x)$. To prove the first
 relation, let us fix a positive number $\epsilon>0$. It follows
 from Proposition \ref{glownapropozycja} (setting $\alpha=\epsilon$) that there exists $M_1>M$,
 such that $x>M_1$ implies, that there exists a number $\theta<\epsilon$
($\theta$ depending on $x$) such that $h^{*}_{+}(x)=\theta \cdot
x$. But this means, that $$\frac{h^{*}_{+}(x)}{x}<\epsilon$$ for
$x>M_1$.  The proof for $h^{*}_{-}$ is similar.

\end{proof}
\vspace{3mm}

 Now we can prove a theorem on the order of magnitude of the length
of lenses S$_k$ using the Proposition \ref{Prop o male}. First we
shall prove the following lemma about sequences tending to
$+\infty$.

\vspace{3mm}

\begin{lemma}\label{ciagi}
 {\it  Suppose that we have four sequences
$(x^{-}_k)_1^{\infty}$,$(x^{+}_k)_1^{\infty}$,$(z_k)_1^{\infty}$,
and $(e_k)_1^{\infty}$ such that:

\be 0<x^{-}_k\leq e_k<e_{k+1}\leq x^{+}_k, \ee
\be x^{-}_k\leq  z_k \leq x^{+}_k, \ee
\be \lim_{k\rightarrow \infty} e_k=+\infty,\ee 
\be \lim_{k\rightarrow
\infty}\frac{x^{+}_k-x^{-}_k}{z_k}=0.\ee

Then $$\lim_{k\rightarrow \infty}\frac{e_{k+1}-e_k}{e_k} =0.$$}
\end{lemma}
\vspace{3mm}

\begin{proof}

 From (60) and (62) we deduce that: $$\lim_{k\rightarrow \infty}x^{+}_k = +
\infty.$$ It must be also $$\lim_{k\rightarrow \infty}x^{-}_k = +
\infty.$$Indeed, suppose that there exists an infinite subset
$\mathbb L\subset \mathbb N$ and a constant $K>0$ such that $0\leq
x^{-}_n\leq K$ for $n\in \mathbb L$. Then for $n\in \mathbb L$ we
have:

$$0\leq \frac{x^{+}_n-K}{z_n}\leq\frac{x^{+}_n-x^{-}_n}{z_n}$$
Hence by (63) $$ \frac{x^{+}_n-K}{z_n}\rightarrow 0,n\in \mathbb
L.$$ This implies that $\lim_{n\in \mathbb L}z_n = +\infty$. In
consequence $$\lim_{n\in \mathbb L}\frac{x^{+}_n}{z_n} = 0,$$ thus
there exists $n\in \mathbb L$ such that $x^{+}_n<z_n$, but this is
impossible.

From the inequality $$\frac{x^{+}_k-x^{-}_k}{x^{+}_k}\leq
\frac{x^{+}_k-x^{-}_k}{z_k}$$ we deduce that
$$\lim_{k\rightarrow +\infty}\frac{x^{-}_k}{x^{+}_k}=1$$ and this gives
$$\lim_{k\rightarrow +\infty}\frac{x^{+}_k-x^{-}_k}{x^{-}_k}=0.$$
But
$$\frac{x^{+}_k-x^{-}_k}{e_k}\leq
\frac{x^{+}_k-x^{-}_k}{x^{-}_k}$$then
$$\lim_{k\rightarrow \infty}\frac{x^{+}_k-x^{-}_k}{e_k}=0.$$
Since$$\frac{e_{k+1}-e_k}{e_k}\leq \frac{x^{+}_{k}-x^{-}_k}{e_k}$$
then$$\lim_{k\rightarrow \infty}\frac{e_{k+1}-e_k}{e_k} =0,$$and
this ends the proof of  Lemma \ref{ciagi}.

\end{proof}
\vspace{3mm}
\begin{lemma}\label{graph}
{ \it The graph of the function $\pi^{*}$ lies between the graphs
of the functions $Li-\varepsilon$ and $Li+\varepsilon$}.
\end{lemma}
\vspace{3mm}

\begin{proof}

 Suppose the opposite. Then there exist two consecutive prime
numbers $p_n$ and $p_{n+1}$, such that the points $A=(p_n,n)$ and
$B=(p_{n+1},n+1)$ lies between $Li-\varepsilon$ and
$Li+\varepsilon$ and the segment $[A;B]$ cuts the graph of
$Li-\varepsilon$ or $Li+\varepsilon$. But the subgraph of
$Li+\varepsilon$ is convex, then $[A;B]$ cuts only the graph of
$Li-\varepsilon$. This means, that there exists a point $x\in
(p_n,p_{n+1})$ such that the point $X=(x,n)$ lies below the graph
of $Li-\varepsilon$. But $X=(x,\pi(x))$, then from the definition
of the error term, $X$ lies between the graphs of $Li-\varepsilon$
and $Li+\varepsilon$. This ends the proof of Lemma \ref{graph}.
\end{proof}
\vspace{3mm}

\begin{lemma}\label{lemat29}
 {\it Let $S_k$ be a lens defined by the extremal prime numbers
$e_k$ and $e_{k+1}$. Then the straight line joining the points
$U=(e_k,\pi(e_k))$ and $V=(e_{k+1},\pi(e_{k+1}))$ cannot cut the
graph of $Li-\varepsilon$ in two distinct points.}
\end{lemma}
\vspace{3mm}

\begin{proof}

 This follows from the Lemma \ref{graph} since, by the definition of
extremal points, all the graph of $\pi^{*}$ lies below the
straight line joining the points $U$ and $V$.
\end{proof}
\vspace{3mm}

The main theorem of this section is the following:

\vspace{3mm}
\begin{theorem}
{\it With the notations as above if the Riemann Conjecture is
true, then $$\lim_{k\rightarrow +\infty}\frac{e_{k+1}}{e_k} =
1.$$}
\end{theorem}
\vspace{3mm}

\begin{proof}
Let $U$ and $V$ be as in Lemma \ref{lemat29}. Take the straight
line $l(U,V)$ joining $U$ and $V$  and translate it to the
position $l^{*}$ where the straight line $l^{*}$ is parallel to
$l(U,V)$ and tangent to the graph of $Li-\varepsilon$. This line
$l^{*}$ cuts the graph of $Li+\varepsilon$ in points $U^{*}$ and
$V^{*}$, whose first coordinates are $x^{-}_k$ and $x^{+}_k$
respectively, and the tangent point is $z_k$. It is not hard to
check, that the sequences
$(x^{-}_k)_1^{\infty}$,$(x^{+}_k)_1^{\infty}$,$(z_k)_1^{\infty}$,
and $(e_k)_1^{\infty}$ satisfy the assumptions of  Lemma
\ref{ciagi}. Then this ends the proof of the theorem.

\end{proof}
\vspace{3mm}
 We have an equivalent formulation.

\vspace{3mm}

\begin{corollary}
 { \it The length of lenses $x\rightarrow
 S(x)$ satisfies the equality $S(x)=o(x)$.}

\end{corollary}

\section{Part III}

\subsection{Final remarks}

It is natural to ask if one can prove the results like  Theorem 30
or Corollary 31 without assuming the Riemann Hypothesis. Maybe
this is possible, but it seems, that the method used in this paper
is insufficient. In particular an analogous argumentation applied
to $L(x)=\frac{x}{\ln x}$ and $\varepsilon(x)=
C\cdot\frac{x}{\ln^{2}x}$ gives only $S(x)=O(x)$. I was also not
able to prove  Theorem 30 using $L(x)= Li(x)$ and
$$\varepsilon(x)=O\left(x\cdot \exp\left(\frac{A(\ln x)^{\frac{3}{5}}}{(\ln(\ln
x))^{\frac{1}{5}}}\right)\right).$$ On the other hand for
$L(x)=Li(x)$ the error term $\varepsilon(x)=O(x^{\alpha}\cdot
(\ln^{k}x))$ ( $\alpha
>\frac{1}{2}$ and $k\in \mathbb Z$) is sufficient.

If one assumes the Riemann hypothesis, then some naive
argumentation leads to the equality like $S(x)= O(\sqrt{x}
\ln^{2}x)$, which seems to be supported by the experimental data.
This may suggest, that the problem of determining the right order
of magnitude of $S(x)$ at infinity is near to the problem of
determining the right order of magnitude of the difference
$|Li(x)-\pi(x)|$.

I have no idea about "the small gaps between extremal primes". As
it was mentioned in Part I, Question \ref{pytanie}, the small gaps
between extremal primes -i.e. the small $S_k$- may occur, but the
theorems like for example $$\liminf\frac{e_{k+1}-e_k}{\ln e_k}=0$$
or at least $$\liminf\frac{e_{k+1}-e_k}{\sqrt{e_k}}=0$$ seems to
be out of reach.

As it was mentioned in Introduction, Montgomery and Wagon in
\cite{MonWa} considered the function $M(x)=x\rightarrow
\frac{x}{\pi(x)}$. I used an analogous algorithm as in Proposition
\ref{rekurencja} to obtain about 1500 "another"  extremal prime
numbers, $(m_k)_1^{\infty}$  "generated" by the function $M(x)$
instead of $\pi(x)$. Generated by $M(x)$ means, that the points
$(m_k,M(m_k))$ are extremal points of the convex hull of the
subgraph of the function $M(x)$. Clearly $(m_k)_1^{\infty}$ and
$(e_k)_1^{\infty}$ are not the same sequences, there are many
differences, but on the other hand they behave (in asymptotic
sense) similarly.

  \begin{figure}[h]
      \begin{center}
        \scalebox{0.6}{\includegraphics{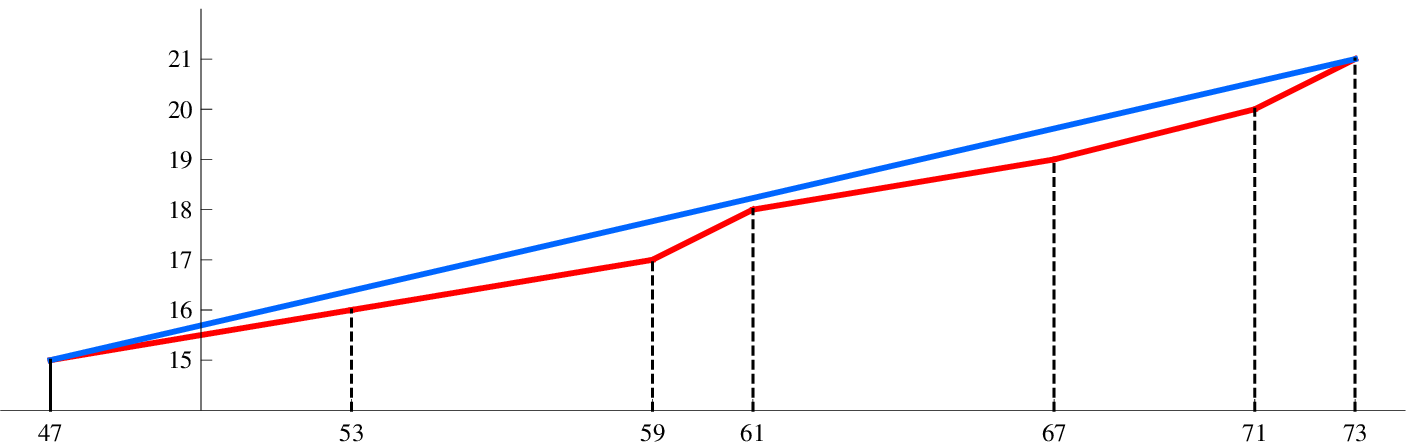}}
      \end{center}
      \end{figure}

  \begin{figure}[h]
      \begin{center}
        \scalebox{0.5}{\includegraphics{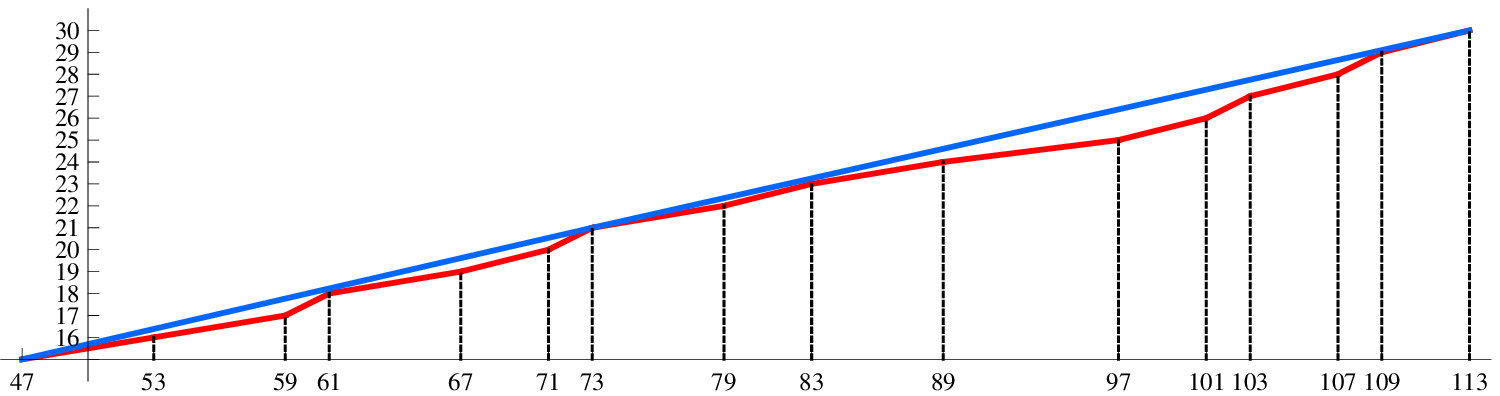}}
      \end{center}
      \end{figure}

         \begin{figure}[h]
      \begin{center}
        \scalebox{0.6}{\includegraphics{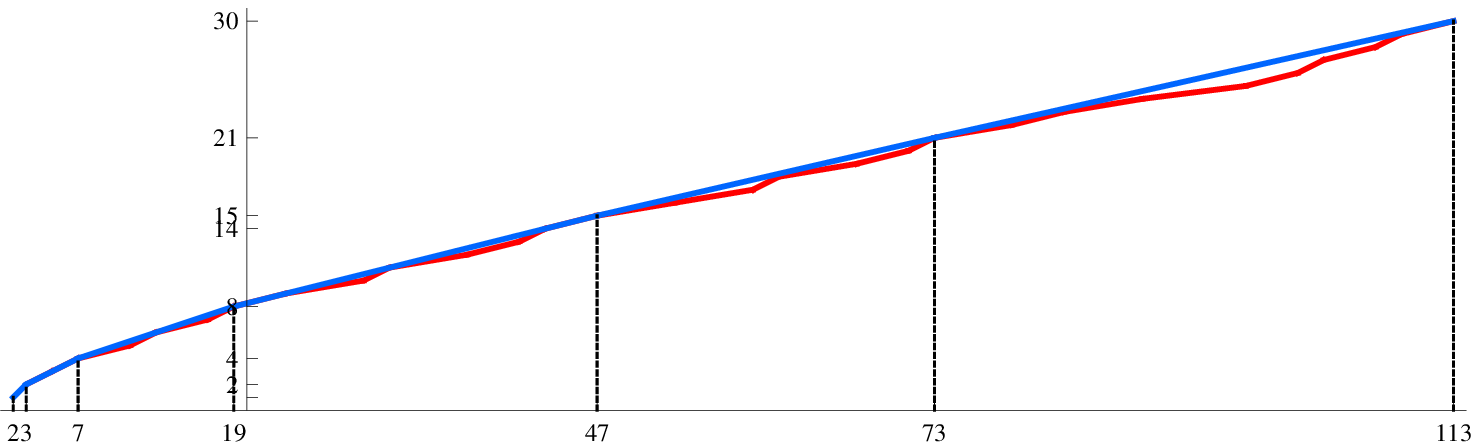}}
      \end{center}
      \end{figure}

\end{document}